\documentclass[12pt]{amsart}
\usepackage[english]{babel}
\usepackage[T1]{fontenc}
\usepackage{fullpage}
\usepackage{color}
\usepackage{amsmath,amscd,amsthm,amsfonts,latexsym}

\theoremstyle{plain}
\newtheorem{theorem}                 {Theorem}      [section]
\newtheorem{proposition}  [theorem]  {Proposition}

\newtheorem{lemma}        [theorem]  {Lemma}

\theoremstyle{definition}
\newtheorem{example}      [theorem]  {Example}
\newtheorem{remark}       [theorem]  {Remark}
\newtheorem{definition}   [theorem]  {Definition}

\DeclareMathOperator{\cst}{constant}
\def \r{\mbox{${\mathbb R}$}}

\def \c{\mbox{${\mathbb C}$}}
\def \h{\mbox{${\mathbb H}$}}

\def \s{\mbox{${\mathbb S}$}}
\def \S{\mbox{${\mathrm {SL}}(2,\r)_\tau$}}

\def \n{\mbox{${{\nabla}}^\tau$}}

\begin{document}

\title{Helix surfaces in the special linear group}

\author{S. Montaldo}

\address{Universit\`a degli Studi di Cagliari\\
Dipartimento di Matematica e Informatica\\
Via Ospedale 72\\
09124 Cagliari}
\email{montaldo@unica.it}

\author{I. I. Onnis}
\address{Departamento de Matem\'{a}tica, C.P. 668\\ ICMC,
USP, 13560-970, S\~{a}o Carlos, SP\\ Brasil}
\email{onnis@icmc.usp.br}

\author{A. Passos Passamani}
\address{Departamento de Matem\'{a}tica, C.P. 668\\ ICMC,
USP, 13560-970, S\~{a}o Carlos, SP\\ Brasil}
\email{apoenapp@icmc.usp.br}

\thanks{The first author was supported by P.R.I.N. 2010/11 -- Variet\`a reali e complesse: geometria, topologia e analisi armonica -- Italy and INdAM. The third author was supported by Capes--Brazil}

\begin{abstract}
We characterize helix surfaces (constant angle surfaces) in the special linear group $\mathrm{SL}(2,\r)$. In particular, we give an explicit local description of these surfaces by means of a suitable curve and  a 1-parameter family of isometries of $\mathrm{SL}(2,\r)$.
\end{abstract}

\maketitle

\section{Introduction}

In recent years much work has been done to understand the geometry of surfaces whose unit normal vector field forms a constant angle with a fixed field of directions of the ambient space. These surfaces are called {\em helix surfaces} or {\em constant angle surfaces} and they have been studied in most of the $3$-dimensional geometries. In \cite{CDS07} Cermelli and Di Scala analyzed the case of constant angle surfaces in $\r^3$ obtaining a remarkable relation with a Hamilton-Jacobi equation and showing their application to  equilibrium configurations of liquid crystals. Later, Dillen--Fastenakels--Van der Veken--Vrancken  (\cite{DFVV07}), and  Dillen--Munteanu (\cite{DM09}), classified the surfaces making a constant angle with the $\r$-direction in the product spaces $\s^2\times\r$ and $\h^2\times\r$, respectively. Moreover, helix submanifolds have been studied in higher dimensional euclidean spaces and product spaces in \cite{DSRH10,DSRH09,RH11}.

The spaces $\s^2\times\r$ and $\h^2\times\r$ can be seen as two particular cases of Bianchi-Cartan-Vranceanu spaces (BCV-spaces) which include all 3-dimensional homogeneous metrics whose group of isometries has dimension $4$ or $6$, except for those of constant negative sectional curvature. A crucial feature of BCV-spaces is that they admit a  Riemannian submersion onto a surface of constant Gaussian curvature, called the Hopf fibration, that, in the cases of $\s^2\times\r$ and $\h^2\times\r$, it is the natural projection onto the first factor. Consequently, one can consider the angle $\vartheta$ that the unit normal vector field of a surface in a BCV-space forms with the Hopf vector field, which is, by definition, the vector field tangent to the fibers of the Hopf fibration. This angle $\vartheta$ has a crucial role in the study of surfaces in BCV-spaces as shown by Daniel, in \cite{B}, where he proved that the equations of Gauss and Codazzi are given in terms of the function $\nu=\cos\vartheta$ and that this angle is one of the fundamental invariants for a surface in BCV-spaces.
Consequently, in \cite{FMV11},  the authors considered  the surfaces in a BCV-space for which the angle $\vartheta$ is constant, giving a complete local classification in the case that the BCV-space is the Heisenberg space $\h_3$.

Later, L\'opez--Munteanu, in \cite{LM11}, defined and classified two types of constant angle surfaces in the homogeneous 3-manifold $\mathrm{Sol}_3$, whose isometry group has dimension 3. Also, Montaldo--Onnis, in \cite{MO}, characterized helix surfaces in the 1-parameter family of Berger spheres $\s^3_\epsilon$, with $\epsilon>0$, proving that, locally, a helix surface is  determined by a suitable 1-parameter family of isometries of the Berger sphere and by a geodesic of a $2$-torus in the $3$-dimensional sphere.

This paper is a continuation of our work \cite{MO} and it is devoted to the study and characterization of helix surfaces in the  homogeneous 3-manifold given by the special linear group $\mathrm{SL}(2,\r)$ endowed with a suitable $1$-parameter family $g_{\tau}$ of metrics that we shall describe in section~\ref{preli}.

Our study of helix surfaces in $(\mathrm{SL}(2,\r),g_{\tau})$ will depend on a constant $B:=(\tau^2+1)\cos^2\vartheta-1$, where $\vartheta$ is the constant angle between
the normal to the surface and the Hopf vector field of $\mathrm{SL}(2,\r)$. A similar constant appeared also in the study of helix surfaces in the Berger sphere (see \cite{MO}) but in that case the constant was always positive. Thus we shall divide our study according to the three possibilities: $B>0$, $B=0$ and $B<0$.

\section{Preliminaries}\label{preli}

Let $\r_2^4$ denote the $4$-dimensional pseudo-Euclidean space endowed with the semi-definite inner product of signature $(2,2)$ given by
$$
\langle v,w\rangle= v_1\,w_1+v_2\,w_2-v_3\,w_3-v_4\,w_4\,,\quad v,w\in\r^4.
$$
We identify the special linear group with
$$
\mathrm{SL}(2,\r)=\{(z,w)\in\c^2 \colon |z|^2-|w|^2=1\}=\{v\in\r_2^4 \colon \langle v,v\rangle=1\}\subset\r_2^4
$$
and we shall use the Lorentz model of the hyperbolic plane with constant Gauss curvature $-4$, that is
$$
\h^2(-4)=\{(x,y,z)\in\r^3_1\colon x^2+y^2-z^2=-1/4\},
$$
where $\r^3_1$ is the Minkowski $3$-space. Then the Hopf map $\psi:\mathrm{SL}(2,\r)\to\h^2(-4)$
given by
$$
\psi(z,w)=\frac{1}{2}\,(2z\bar{w},|z|^2+|w|^2)
$$
is a submersion, with circular fibers, and if we put
$$
X_1(z,w)=(iz,iw),\quad X_2(z,w)=(i\bar{w},i\bar{z}),\quad X_3(z,w)=(\bar{w},\bar{z}),
$$
we have that $X_1$ is a vertical vector field while $X_2$, $X_3$ are horizontal. The vector field $X_1$ is called the {\em Hopf vector field}.

We shall endow $\mathrm{SL}(2,\r)$ with the $1$-parameter family of metrics $g_\tau$, $\tau>0$, given by
$$
g_\tau(X_i,X_j)=\delta_{ij},\quad g_\tau(X_1,X_1)=\tau^2,\quad g_\tau(X_1,X_j)=0,\quad i,j\in\{2,3\},
$$
which renders the Hopf map $\psi:(\mathrm{SL}(2,\r),g_{\tau})\to\h^2(-4)$ a Riemannian submersion.

{For those familiar with the notations of Daniel \cite{B}, we point out that $(\mathrm{SL}(2,\r),g_{\tau})$ corresponds to a model for a homogeneous space $E(k, \tau)$ with curvature of the basis $k=-4$ and bundle curvature $\tau>0$.}

With respect to the inner product in $\r_2^4$ the metric $g_{\tau}$ is given by
\begin{equation}\label{eq-def-gtau}
g_{\tau}(X,Y)=-\langle X,Y\rangle+(1+\tau^2)\langle X, X_1\rangle \langle Y, X_1\rangle\,.
\end{equation}
From now on, we denote $(\mathrm{SL}(2,\r),g_\tau)$ with $\S$. Obviously
\begin{equation}\label{eq-basis}
    E_1=-\tau^{-1}\,X_1,\quad
     E_2=X_2,\quad
     E_3=X_3,
\end{equation}
is an orthonormal  basis on $\S$ and the Levi-Civita connection $\n$ of $\S$ is given by {(see, for example, \cite{t})}:
\begin{equation}
\begin{aligned}
&\n_{E_{1}}E_{1}=0,\quad \n_{E_{2}}E_{2}=0,\quad \n_{E_{3}}E_{3}=0,\\
&\n_{E_{1}}E_{2}=-\tau^{-1}(2+\tau^2)E_{3}, \quad \n_{E_{1}}E_{3}=\tau^{-1}(2+\tau^2)E_{2},\\
 &\n_{E_{2}}E_{1}=-\tau E_{3},\quad \n_{E_{3}}E_{1}=\tau E_{2},\quad
 \n_{E_{3}}E_{2}=-\tau E_{1}=-\n_{E_{2}}E_{3}.
\end{aligned}
\label{nabla}
\end{equation}

Finally, we recall that the isometry group of $\S$ is the $4$-dimensional indefinite unitary group $\mathrm{U}_1(2)$
that can be identified with:
$$
\mathrm{U}_1(2)=\{A\in \mathrm{O}_2(4)\colon AJ_1=\pm J_1A\}\,,
$$
where $J_1$ is the complex structure of $\r^4$ defined by
$$
J_1 = \left(\begin{matrix}J & 0 \\ 0 & J\end{matrix}\right)\,,\quad J = \left(\begin{matrix}0 & -1 \\ 1 & 0\end{matrix}\right)\,,
$$
while
$$
\mathrm{O}_2(4)=\{A\in \mathrm{GL}(4,\r)\colon A^t=\epsilon\,A^{-1}\,\epsilon\},\qquad \epsilon=\begin{pmatrix}I&0\\0&-I\end{pmatrix},\quad I=\begin{pmatrix}1&0\\0&1\end{pmatrix}
$$
is the indefinite orthogonal group.

We observe that  $\mathrm{O}_2(4)$ is the group of $4\times 4$ real matrices preserving the semi-definite inner product of $\r_2^4$.

Suppose now we are given a 1-parameter family $A(v)\,, v\in(a,b)\subset\r$, consisting of $4\times 4$ indefinite orthogonal matrices commuting (anticommuting, respectively) with $J_1$. In order to describe explicitly the family $A(v)$, we shall use two product structures of $\r^4$, namely
$$
 J_{2} =\begin{pmatrix}
 0&0 & 0 & 1 \\
 0&0 & 1 & 0 \\
 0&1& 0 & 0 \\
  1&0 & 0 & 0 \\
 \end{pmatrix}\,,\qquad
 J_{3} =\begin{pmatrix}
 0&0 & 1 &0 \\
 0&0 & 0 & -1 \\
 1&0& 0 & 0 \\
0&-1 & 0 & 0 \\
 \end{pmatrix}\,.
 $$
Since $A(v)$ is an indefinite orthogonal matrix, the first row must be a unit vector ${\mathbf r}_1(v)$ of $\r^4_2$ for all $v\in(a,b)$. Thus, without loss of generality, we can take
$$
{\mathbf r}_1(v)=(\cosh\xi_1(v)\cos\xi_2(v), -\cosh\xi_1(v)\sin\xi_2(v), \sinh\xi_1(v)\cos\xi_3(v),-\sinh\xi_1(v)\sin\xi_3(v))\,,
$$
for some real functions $\xi_1,\xi_2$ and $\xi_3$ defined in $(a,b)$. Since $A(v)$ commutes  (anticommutes, respectively) with $J_1$ the second row of $A(v)$ must be ${\mathbf r}_2(v)=\pm J_{1}{\mathbf r}_1(v)$. Now, the four vectors $\{{\mathbf r}_1, J_1{\mathbf r}_1, J_2{\mathbf r}_1,J_3{\mathbf r}_1\}$ form a pseudo-orthonormal basis of $\r^4_2$, thus the third row ${\mathbf r}_3(v)$ of $A(v)$ must be a linear combination of them. Since ${\mathbf r}_3(v)$ is unit and it is orthogonal to both ${\mathbf r}_1(v)$ and $J_1{\mathbf r}_1(v)$, there exists a function $\xi(v)$ such that
$$
{\mathbf r}_3(v)=\cos\xi(v) J_2 {\mathbf r}_1(v)+\sin\xi(v) J_3 {\mathbf r}_1(v)\,.
$$
Finally the fourth row of $A(v)$ is ${\mathbf r}_4(v)=\pm J_1{\mathbf r}_3(v)=\mp\cos\xi(v) J_3 {\mathbf r}_1(v)\pm\sin\xi(v) J_2 {\mathbf r}_1(v)$.
This means that any $1$-parameter family $A(v)$ of $4\times 4$ indefinite orthogonal matrices commuting (anticommuting, respectively) with $J_1$
can be described by four functions $\xi_1,\xi_2,\xi_3$ and $\xi$ as
\begin{equation}\label{eq-descrizione-A}
A(\xi,\xi_1,\xi_2,\xi_3)(v)=
\begin{pmatrix}
{\mathbf r}_1(v)\\
\pm J_1{\mathbf r}_1(v)\\
\cos\xi(v) J_2 {\mathbf r}_1(v)+\sin\xi(v) J_3 {\mathbf r}_1(v)\\
\mp\cos\xi(v) J_3 {\mathbf r}_1(v)\pm\sin\xi(v) J_2 {\mathbf r}_1(v)
\end{pmatrix}\,.
\end{equation}

\section{Constant angle surfaces}
We start this section giving the definition of constant angle surface in  $\S$.

\begin{definition}
We say that a surface in the special linear group $\S$ is a {\it helix surface} or a {\it constant angle surface} if the angle $\vartheta\in [0,\pi)$ between the unit normal vector field and the unit Killing vector field $E_1$ (tangent to the fibers of the Hopf fibration) is constant at every point of the surface.
\end{definition}

Let $M^2$ be an oriented helix surface in $\S$ and let  $N$ be a unit normal vector field. Then, by definition,
$$
|g_{\tau}(E_1,N)|=\cos\vartheta,
$$
for fixed $\vartheta\in[0,\pi/2]$. Note that $\vartheta\neq 0$. In fact, if it were zero then the vector fields $E_2$ and $E_3$   would be tangent to the surface $M^2$, which is absurd since  the horizontal distribution of the Hopf map is not integrable. If $\vartheta=\pi/2$, we have that $E_1$ is always tangent to $M$ and, therefore, $M$ is a Hopf cylinder. Therefore, from now on we assume that the constant angle $\vartheta\neq \pi/2,0$.

The Gauss and Weingarten formulas are
\begin{equation}\label{gauss-wein}\begin{aligned}
    \n_X Y&=\nabla_X Y+\alpha(X,Y),\\
    \n_X N&=-A(X),
    \end{aligned}
\end{equation}
where with $A$ we have indicated  the shape operator of $M$ in $\mathrm{SL}(2,\r)_\tau$, with $\nabla$ the induced Levi-Civita connection on $M$ and by $\alpha$  the second fundamental form of $M$ in $\mathrm{SL}(2,\r)_\tau$.
Projecting $E_1$ onto the tangent plane to $M$ we have $$E_1=T+\cos\vartheta\, N,$$ where $T$ is the tangent part which satisfies $g_\tau(T,T)=\sin^2\vartheta.$

For all $X\in TM$, we have that
\begin{equation}\label{eq1}\begin{aligned}
    \n_X E_1&=\n_X T-\cos\vartheta\,A(X)\\&=\nabla_X T+g_\tau(A(X),T)\,N-\cos\vartheta\,A(X).\end{aligned}
\end{equation}
On the other hand, if $X=\sum X_i E_i$,
\begin{equation}\label{eq2}\begin{aligned}
    \n_X E_1&=\tau\,(X_3 E_2-X_2 E_3)=
    \tau\, X\wedge E_1\\&=\tau\,g_\tau(JX,T)\,N-\tau\,\cos\vartheta JX,
    \end{aligned}
\end{equation}
where $JX=N\wedge X$ denotes the rotation of angle $\pi/2$ on $TM$.
Identifying the tangent and normal component of \eqref{eq1} and \eqref{eq2} respectively, we obtain
\begin{equation}\label{eq3}
    \nabla_X T= \cos\vartheta\,(A(X)-\tau\,JX)
\end{equation}
and
\begin{equation}\label{eq4}
    g_\tau(A(X)-\tau\,JX,T)=0.
\end{equation}

\begin{lemma}\label{princ}
Let $M^2$ be an oriented helix surface in $\S$ with constant angle $\vartheta$. Then, we have the followings properties.
\begin{itemize}
  \item[(i)] With respect to the basis $\{T,JT\}$, the matrix associates to the shape operator $A$ takes the form
$$
A=\left(
\begin{array}{cc}
0 & -\tau \\
-\tau & \lambda\\
 \end{array}
 \right),
  $$
  for some function $\lambda$ on $M$.
  \item[(ii)] The Levi-Civita connection $\nabla$ of $M$ is given by
  $$\nabla_T T=-2\tau\cos\vartheta\, JT,\qquad \nabla_{JT} T=\lambda\cos\vartheta\, JT,$$
  $$\nabla_T JT=2\tau\cos\vartheta\, T,\qquad \nabla_{JT} JT=-\lambda\cos\vartheta\, T.$$
  \item[(iii)] The Gauss curvature of $M$ is constant and satisfies $$K=-4(1+\tau^2)\,\cos^2\vartheta.$$
  \item[(iv)] The function $\lambda$ satisfies the equation
  \begin{equation}\label{lambda}
    T \lambda+\lambda^2\,\cos\vartheta+4B\,\cos\vartheta=0,
  \end{equation}
where $B:=(\tau^2+1)\cos^2\vartheta-1.$
\end{itemize}
\end{lemma}
\begin{proof} Point (i) follows directly from \eqref{eq4}.
From \eqref{eq3} and using $$g_\tau(T,T)=g_\tau(JT,JT)=\sin^2\vartheta,\quad g_\tau(T,JT)=0,$$ we obtain (ii). From the Gauss equation in $\S$ (we refer to the equation in Corollary~3.2 of \cite{B} with $\nu=\cos\theta$ and $k=-4$),  and (i), we have that the Gauss curvature of $M$ is given by
$$\begin{aligned}K&=\det A+\tau^2-4(1+\tau^2)\,\cos^2\vartheta\\ &=-4(1+\tau^2)\,\cos^2\vartheta.\end{aligned}$$
 Finally,  \eqref{lambda} follows from the Codazzi equation (see \cite{B}): $$\nabla_X A(Y)-\nabla_Y A(X)-A[X,Y]=-4(1+\tau^2)\,\cos\vartheta\,(g_\tau(Y,T)X-g_\tau(X,T)Y)),$$ putting $X=T$, $Y=JT$ and using (ii).
 {In fact, it is easy to check that}
 $$-4(1+\tau^2)\,\cos\vartheta\,(g_\tau(JT,T)T-g_\tau(T,T)JT))=4(1+\tau^2)\cos\vartheta\,\sin^2\vartheta\, JT$$ and
 $$\begin{aligned}&\nabla_T A(JT)-\nabla_{JT} A(T)-A[T,JT]\\&=
 \nabla_T (-\tau\, T+\lambda\, JT)-\nabla_{JT} (-\tau\, JT)-A(2\tau \cos\vartheta\,T-\lambda \cos\vartheta\, JT)\\&=
 (4\tau^2\cos\vartheta+T(\lambda)+\lambda^2\,\cos\vartheta)\,JT.
 \end{aligned}$$
\end{proof}
As $g_\tau (E_1,N)=\cos\vartheta$, there exists a smooth function $\varphi$ on $M$ such that
$$N=\cos\vartheta E_1+\sin\vartheta\cos\varphi\,E_2+\sin\vartheta\sin\varphi\,E_3.$$
Therefore
\begin{equation}\label{eq:def-T}T=E_1-\cos\vartheta\,N=\sin\vartheta\,[\sin\vartheta\,E_1-\cos\vartheta\cos\varphi\,E_2-\cos\vartheta\sin\varphi\,E_3]
\end{equation}
and $$JT=\sin\vartheta\,(\sin\varphi\,E_2-\cos\varphi\,E_3).$$
Also
\begin{equation}\begin{aligned}\label{eqTJ}
    A(T)&=-\n_T N=(T \varphi-\tau^{-1}(2+\tau^2)\,\sin^2\vartheta+\tau\cos^2\vartheta)\,JT,\\
     A(JT)&=-\n_{JT} N=(JT\varphi)\,JT-\tau\,T.
\end{aligned}\end{equation}
Comparing \eqref{eqTJ} with (i) of Lemma~\ref{princ}, it results that
\begin{equation}\label{eqTJ1}\left\{\begin{aligned}
JT\varphi&=\lambda,\\
T\varphi&=-2\tau^{-1}\,B.
\end{aligned}
\right.
\end{equation}
We observe that, as $$[T,JT]=\cos\vartheta\,(2\tau\, T-\lambda\,JT),$$ the compatibility condition of system~\eqref{eqTJ1}:
$$(\nabla_T JT-\nabla_{JT} T)\varphi=[T,JT]\varphi=T (JT\varphi)-JT (T\varphi)$$ is equivalent to
\eqref{lambda}.

We now choose local coordinates $(u,v)$ on $M$ such that
\begin{equation}\label{eq:local-coordinates}
\partial_u=T.
\end{equation}
 Also, as $\partial_v$ is tangent to $M$, it can be written in the form
 \begin{equation}\label{eq-definition-Fv}
 \partial_v=a\,T+b\,JT\,,
 \end{equation}
  for certain functions $a=a(u,v)$ and $b=b(u,v)$. As
$$0=[\partial_u,\partial_v]=(a_u+2\tau b\cos\vartheta)\,T+(b_u-b\lambda\cos\vartheta)\,JT,$$ then
\begin{equation}\label{eqab}\left\{\begin{aligned}
a_u&=-2\tau b\cos\vartheta,\\
b_u&=b\lambda\cos\vartheta.
\end{aligned}
\right.
\end{equation}
Moreover, the equation \eqref{lambda} of Lemma~\ref{princ} can be written as
\begin{equation}\label{eq-lambda-du}
\lambda_u+\cos\vartheta\,\lambda^2+4B\,\cos\vartheta=0.
\end{equation}
Depending on the value of $B$, by integration of \eqref{eq-lambda-du}, we have the following three possibilities.
\begin{itemize}
\item [(i)] If $B=0$
$$\lambda(u,v)=\frac{1}{u\,\cos\vartheta+\eta(v)},$$
for some smooth function $\eta$ depending on $v$. Thus the solution of system~\eqref{eqab} is given by
\begin{equation}\left\{\begin{aligned}\nonumber
a(u,v)&=-\tau\,u\,\cos\vartheta\,(u\,\cos\vartheta+2\,\eta(v)),\\
b(u,v)&=u\,\cos\vartheta+\eta(v).
\end{aligned}
\right.
\end{equation}
\item[(ii)]
If $B>0$
\begin{equation}\label{eqlambda}\nonumber
    \lambda(u,v)=2\,\sqrt{B}\tan (\eta(v)-2\cos\vartheta \sqrt{B}\,u),
\end{equation}
for some smooth function $\eta$ depending on $v$ and system~\eqref{eqab} has the solution
$$\left\{\begin{aligned}
a(u,v)&=\frac{\tau}{\sqrt{B}}\sin (\eta(v)-2\cos\vartheta \sqrt{B}\,u),\\
b(u,v)&=\cos(\eta(v)-2\cos\vartheta \sqrt{B}\,u).
\end{aligned}
\right.
$$
\item[(iii)] If $B<0$
\begin{equation}\label{eqlambda2}\nonumber
    \lambda(u,v)=2\,\sqrt{-B}\tanh ({\eta}(v)+2\cos\vartheta \sqrt{-B}\,u),
\end{equation}
for some smooth function ${\eta}$ depending on $v$.
Solving the system~\eqref{eqab}, we have
$$\left\{\begin{aligned}
a(u,v)&=-\frac{\tau}{\sqrt{-B}}\sinh ({\eta}(v)+2\cos\vartheta \sqrt{-B}\,u),\\
b(u,v)&=\cosh({\eta}(v)+2\cos\vartheta \sqrt{-B}\,u).
\end{aligned}
\right.
$$
\end{itemize}
Moreover, in the case (i) the system~\eqref{eqTJ1} becomes
\begin{equation}\label{eqTJ2i}\left\{\begin{aligned}
\varphi_u&=0,\\
\varphi_v&=1,
\end{aligned}
\right.
\end{equation}
and so $\varphi(u,v)=v+c$, $c\in\r$. In the cases (ii) and (iii), the system~\eqref{eqTJ1} becomes
\begin{equation}\label{eqTJ2}\left\{\begin{aligned}\nonumber
\varphi_u&=-2\tau^{-1}\,B,\\
\varphi_v&=0,
\end{aligned}
\right.
\end{equation}
of which the general solution is given by
\begin{equation}\label{eqphi2}
    \varphi(u,v)=-2\tau^{-1}B\,u+c,
\end{equation}
where $c$ is a real constant.

With respect to the local coordinates $(u,v)$ chosen above, we have the following characterization of the position vector of a helix surface.

\begin{proposition}
Let $M^2$ be a helix surface in $\S\subset\r_2^4$ with constant angle $\vartheta$.  Then, with respect to the local coordinates $(u,v)$ on $M$ defined in \eqref{eq:local-coordinates},  the position vector $F$ of $M^2$ in $\r^4_2$ satisfies the following equation:
\begin{itemize}
\item [(a)] if $B=0$,
\begin{equation}\label{eqquarta1}
\frac{\partial^2F}{\partial u^2}=0,
\end{equation}
\item [(b)] if $B\neq 0$,
\begin{equation}\label{eqquarta}
\frac{\partial^4F}{\partial u^4}+(\tilde{b}^2-2\tilde{a})\,\frac{\partial^2F}{\partial u^2}+\tilde{a}^2\,F=0,
\end{equation}
where
\begin{equation}\label{eq:value-a-b}
\tilde{a}=-\tau^{-2}\sin^2\vartheta\, B\,, \qquad \tilde{b}=-2 \tau^{-1}\, B.
\end{equation}
\end{itemize}
\end{proposition}

\begin{proof}
Let $M^2$ be a helix surface and let $F$ be the position vector of $M^2$ in $\r_2^4$. Then, with respect to the local coordinates $(u,v)$ on $M$ defined in \eqref{eq:local-coordinates}, we can write $F(u,v)=(F_1(u,v),\dots,F_4(u,v))$. By definition, taking into account \eqref{eq:def-T}, we have that
$$\begin{aligned}F_u&=(\partial_uF_1,\partial_uF_2,\partial_uF_3,\partial_uF_4)=T\\
&=\sin\vartheta\,[\sin\vartheta\,{E_1}_{|F(u,v)}-\cos\vartheta\cos\varphi\,{E_2}_{|F(u,v)}-\cos\vartheta\sin\varphi\,{E_3}_{|F(u,v)}]\,.
\end{aligned}$$
Using the expression of $E_1$, $E_2$ and $E_3$ with respect to the coordinates vector fields of $\r^4_2$,  we obtain
\begin{equation}\label{eqprime}\left\{\begin{aligned}
\partial_uF_1&=\sin\vartheta\,(\tau^{-1}\sin\vartheta\,F_2-\cos\vartheta\cos\varphi\,F_4-\cos\vartheta\sin\varphi\,F_3),\\
\partial_uF_2&=-\sin\vartheta\,(\tau^{-1}\sin\vartheta\,F_1+\cos\vartheta\cos\varphi\,F_3-\cos\vartheta\sin\varphi\,F_4),\\
\partial_uF_3&=\sin\vartheta\,(\tau^{-1}\sin\vartheta\,F_4-\cos\vartheta\cos\varphi\,F_2-\cos\vartheta\sin\varphi\,F_1),\\
\partial_uF_4&=-\sin\vartheta\,(\tau^{-1}\sin\vartheta\,F_3+\cos\vartheta\cos\varphi\,F_1-\cos\vartheta\sin\varphi\,F_2).\\
\end{aligned}
\right.
\end{equation}
Therefore, if $B=0$, taking the derivative of \eqref{eqprime} with respect to $u$  and using \eqref{eqTJ2i}, we obtain that $F_{uu}=0$.

If $B\neq 0$, taking the derivative of \eqref{eqprime} with respect to $u$  and using \eqref{eqphi2}, we find two constants $\tilde{a}$ and $\tilde{b}$ such that
\begin{equation}\label{eqsegunda}\left\{\begin{aligned}
(F_1)_{uu}&=\tilde{a}\,F_1+\tilde{b}\,(F_2)_u,\\
(F_2)_{uu}&=\tilde{a}\,F_2-\tilde{b}\,(F_1)_u,\\
(F_3)_{uu}&=\tilde{a}\,F_3+\tilde{b}\,(F_4)_u,\\
(F_4)_{uu}&=\tilde{a}\,F_4-\tilde{b}\,(F_3)_u,\\
\end{aligned}
\right.
\end{equation}
where $$\tilde{a}=\frac{\tau^{-1}\sin^2\vartheta}{2}\varphi_u=-\tau^{-2}\sin^2\vartheta B, \qquad \tilde{b}=\varphi_u.$$ Finally, taking twice the derivative of \eqref{eqsegunda} with respect to $u$ and using  \eqref{eqprime} and \eqref{eqsegunda} in the derivative we obtain the desired equation \eqref{eqquarta}.
\end{proof}

\begin{remark}\label{re-value-fuu}
As $\langle F,F\rangle=1$, using \eqref{eqquarta}, \eqref{eqprime} and \eqref{eqsegunda},  we find that the position vector $F(u,v)$  and its derivatives  must satisfy the relations:
\begin{equation}\label{eq:Fprocuct}
\begin{array}{lll} \langle F,F\rangle=1\,,&\langle F_u,F_u\rangle=\tilde{a},&  \langle F,F_u\rangle=0,\\
\langle F_u,F_{uu}\rangle=0\,,& \langle F_{uu},F_{uu}\rangle=D\,,& \langle F,F_{uu}\rangle=-\tilde{a},\\
\langle F_u,F_{uuu}\rangle=-D \,,&
\langle F_{uu},F_{uuu}\rangle=0\,,&  \langle F,F_{uuu}\rangle=0,\\ \langle F_{uuu},F_{uuu}\rangle=E,\,&\qquad &
\end{array}
\end{equation}
where $$D=\tilde{a}\,\tilde{b}^2-3\tilde{a}^2,\qquad E=(\tilde{b}^2-2\tilde{a})\,D-\tilde{a}^3.$$
In addition, as
$$
J_1F(u,v)={X_1}_{|F(u,v)}=-\tau\,{E_1}_{|F(u,v)}=-\tau\,(F_u+\cos\vartheta\,N),
$$
using \eqref{eqquarta}--\eqref{eq:Fprocuct},  we obtain the following identities
\begin{equation}\label{eq-fu-jf-main}
\begin{aligned}
&\langle J_1F,F_u\rangle=-\tau^{-1}\sin^2\vartheta,\\&\langle J_1 F,F_{uu}\rangle=0\,,\\
&\langle F_u,J_1F_{uu}\rangle=\tilde{a}\,(\tilde{b}-\tau^{-1}\,\sin^2\vartheta):=I\,,\\
&\langle J_1F_u,F_{uuu}\rangle=0\,,\\
&\langle J_1F_u,F_{uu}\rangle+\langle J_1F,F_{uuu}\rangle=0\,,\\
&\langle J_1F_{uu},F_{uuu}\rangle+\langle J_1F_u,F_{uuuu}\rangle=0\,.
\end{aligned}
\end{equation}
 \end{remark}

Using Remark~\ref{re-value-fuu}
we can prove the following proposition that gives the conditions under which an immersion defines a helix surface.

\begin{proposition}\label{pro-viceversa}
Let $F:\Omega\to\S\subset\r_2^4$ be an immersion from an open set $\Omega\subset\r^2$, with local coordinates $(u,v)$,
such that the projection of $E_1=-\tau^{-1}J_1F$ to the tangent space of $F(\Omega)\subset\S$ is $F_u$.
Then $F(\Omega)\subset\S$ defines a helix surface of constant angle $\vartheta$ if and only if
\begin{equation}\label{viceversa1}
g_{\tau}(F_u,F_u)=g_{\tau}(E_1,F_u)=\sin^2\vartheta\,,
\end{equation}
and
\begin{equation}\label{viceversa2}
g_{\tau}(F_v,E_1)-g_{\tau}(F_u,F_v)=0\,.
\end{equation}
\end{proposition}
\begin{proof}
Suppose that $F$ is a helix surface of constant angle $\vartheta$. Then
\begin{eqnarray*}
 g_{\tau}(F_u,F_u)&=&-\langle F_u,F_u\rangle+(1+\tau^2) \langle F_u,J_1F\rangle^2\\
 &=&\tau^{-2}\sin^2\vartheta B+(1+\tau^2)(\tau^{-2}\sin^4\vartheta)\\
 &=&\sin^2\vartheta\,.
\end{eqnarray*}
Similarly
\begin{eqnarray*}
 g_{\tau}(E_1,F_u)&=&\tau^{-1}\langle J_1F,F_u\rangle-\tau^{-1}(1+\tau^2) \langle J_1F,F_u\rangle  \langle J_1F,J_1F\rangle\\
 &=&\tau^{-1}\langle J_1F,F_u\rangle[1-(1+\tau^2)]=\sin^2\vartheta\,.
\end{eqnarray*}
Finally, using \eqref{eq-definition-Fv}, we have
\begin{eqnarray*}
 g_{\tau}(F_v,E_1)-g_{\tau}(F_u,F_v)&=&-\frac{a}{\tau} g_{\tau}(F_u,J_1F)-\frac{b}{\tau} g_{\tau}(J_1F_u,J_1F)-a g_{\tau}(F_u,F_u)
 -b g_{\tau}(J_1F_u,F_u)\\
 &=& a \sin^2\vartheta - 0  - a \sin^2\vartheta-0=0 \,.
\end{eqnarray*}
For the converse, put
$$
T_2=F_v-\frac{g_{\tau}(F_v,F_u)F_u}{g_{\tau}(F_u,F_u)}\,.
$$
Then, if we denote by $N$ the unit normal vector field to the surface $F(\Omega)$, $\{F_u,T_2,N\}$ is
an orthogonal bases of the tangent space of $\S$ along the surface $F(\Omega)$.
Now, using \eqref{viceversa2}, we get $g_{\tau}(E_1,T_2)=0$, thus $E_1=a\, F_u+c\, N$. Moreover,
using \eqref{viceversa1} and that $g_{\tau}(E_1,F_u)=a\, g_{\tau}(F_u,F_u)$, we conclude that $a=1$.
Finally,
$$
1=g_{\tau}(E_1,E_1)=g_{\tau}(F_u+c\, N,F_u+c\, N)=\sin^2\vartheta+c^2\,,
$$
which implies that $c^2=\cos^2\vartheta$.
Thus the angle between $E_1$ and $N$ is
$$
g_{\tau}(E_1,N)=g_{\tau}(F_u+\cos\vartheta N,N)=\cos\vartheta.
$$
\end{proof}

\section{The case $B=0$}
\begin{theorem}
Let $M^2$ be a helix surface in the $\S\subset\r_2^4$ with constant angle $\vartheta$ such that $B=0$. Then $\cos\vartheta=\frac{1}{\sqrt{1+\tau^2}}$ and, locally,
the position vector of $M^2$ in $\r^4_2$, with respect to the local coordinates $(u,v)$ on $M$ defined in \eqref{eq:local-coordinates}, is given by
\begin{equation}\label{eq-def-F-B-0}
F(u,v)=A(v)\Big(1,-\frac{\tau\,u}{1+\tau^2},\frac{\tau\,u}{1+\tau^2},0\Big)\,,
\end{equation}
where $A(v)=A(\xi,\xi_1,\xi_2,\xi_3)(v)$ is a 1-parameter family of $4\times 4$ indefinite orthogonal matrices commuting with $J_1$, as described in \eqref{eq-descrizione-A}, with
\begin{equation}\begin{aligned}\label{eq-alpha123-b-0}
 &[\xi'(v)+\xi_2'(v)+\xi_3'(v)]\,\sin(\xi_2(v)-\xi_3(v))\,\sinh(2\xi_1(v))-\\&2(\xi'(v)-\xi_3'(v))\,\sinh ^2\xi_1(v)+2\,[\xi_1'(v)\,\cos(\xi_2(v)-\xi_3(v))+\xi_2'(v)\,\cosh^2\xi_1(v)]=0.
  \end{aligned}
\end{equation}
Conversely, a parametrization
$$
F(u,v)=A(v)\Big(1,-\frac{\tau\,u}{1+\tau^2},\frac{\tau\,u}{1+\tau^2},0\Big)\,,
$$
with $A(v)$ as above, defines a helix surface in the special linear group with constant angle $\vartheta=\arccos\frac{1}{\sqrt{1+\tau^2}}$.
\end{theorem}
\begin{proof}
Since $B=0$ we have immediately that $\cos^2\vartheta=1/({1+\tau^2})$.
Integrating \eqref{eqquarta1}, we obtain that
\begin{equation}\label{eq-Fh1h2}
F(u,v)=h^1(v)+u\,h^2(v),
\end{equation}
where $h^i(v)$, $i=1,2$, are vector fields  in $\r^4_2$, depending only on $v$.

Evaluating in $(0,v)$ the identities:
$$
\begin{aligned}
&\langle F,F\rangle=1,\quad\langle F_u,F_u\rangle=0,\\
&\langle F,F_u\rangle=0,\quad \langle J_1F,F_u\rangle=-\tau^{-1}\,\sin^2\vartheta=-\frac{\tau}{1+\tau^2},
\end{aligned}
$$
it results that
\begin{equation}\label{eqh}
\begin{aligned}&\langle h^1(v),h^1(v)\rangle=1,\quad\langle h^1(v),h^2(v)\rangle=0,\\
&\langle h^2(v),h^2(v)\rangle=0,\quad \langle J_1h^1(v),h^2(v)\rangle=-\frac{\tau}{1+\tau^2}.
\end{aligned}
\end{equation}
Moreover, using \eqref{eqprime} in $(0,v)$, we have that
$$
h^2(v)=-\frac{\tau}{1+\tau^2}\,(J_1h^1(v)-h^3(v)),
$$
where $h^3(v)$ is a vector field of $\r_2^4$ satisfying
\begin{equation}\label{eqh1}
\langle h^3(v),h^3(v)\rangle=-1,\quad \langle h^1(v),h^3(v)\rangle=0,\quad \langle J_1h^1(v),h^3(v)\rangle=0\,.
\end{equation}
Consequently, if we fix the orthonormal basis $\{\hat{E}_i\}_{i=1}^4$ of $\r^4_2$ given by
$$
\hat{E}_1=(1,0,0,0)\,,\quad \hat{E}_2=(0,1,0,0)\,,\quad \hat{E}_3=(0,0,1,0)\,,\quad \hat{E}_4=(0,0,0,1)\,,
$$
there must exists a 1-parameter family of matrices $A(v)\in \mathrm{O}_2(4)$, with $J_1\,A(v)=A(v)\,J_1$, such that
$$
h^1(v)=A(v)\hat{E}_1,\quad J_1h^1(v)=A(v)\hat{E}_2,\quad h^3(v)=A(v)\hat{E}_3,\quad J_1h^3(v)=A(v)\hat{E}_4.
$$
Then \eqref{eq-Fh1h2} becomes
$$
F(u,v)=h^1(v)-\frac{\tau\,u}{1+\tau^2}\,(J_1h^1(v)-h^3(v))=A(v)\Big(1,-\frac{\tau\,u}{1+\tau^2},\frac{\tau\,u}{1+\tau^2},0\Big)\,.
$$
Finally, the $1$-parameter family $A(v)$ depends, according to \eqref{eq-descrizione-A},
on four functions $\xi_1(v)$, $\xi_2(v)$, $\xi_3(v)$ and $\xi(v)$ and, in this case, condition \eqref{viceversa2} reduces to
$\langle F_u, F_v\rangle=0$ which is equivalent to \eqref{eq-alpha123-b-0}.

For the converse, let
$$
F(u,v)=A(v)\Big(1,-\frac{\tau\,u}{1+\tau^2},\frac{\tau\,u}{1+\tau^2},0\Big)\,,
$$
be a parametrization where $A(v)=A(\xi(v),\xi_1(v),\xi_2(v),\xi_3(v))$ is a $1$-parameter family of indefinite orthogonal matrices
with functions  $\xi(v),\xi_1(v),\xi_2(v),\xi_3(v)$ satisfying \eqref{eq-alpha123-b-0}. Since $A(v)$ satisfies \eqref{eq-alpha123-b-0}, then $F$ satisfies
\eqref{viceversa2}, thus, in virtue of Proposition~\ref{pro-viceversa}, we only have to show that \eqref{viceversa1} is satisfied for some constant angle $\vartheta$.
For this we put
$$
\gamma(u)=\Big(1,-\frac{\tau\,u}{1+\tau^2},\frac{\tau\,u}{1+\tau^2},0\Big)\,.
$$
Now, using \eqref{eq-def-gtau} and taking into account that $A(v)$ commutes with $J_1$, we get
\begin{eqnarray*}
g_{\tau}(F_u,F_u)&=&-\langle A(v)\gamma'(u),A(v)\gamma'(u)\rangle+(1+\tau^2)\langle A(v)\gamma'(u),J_1A(v)\gamma(u)\rangle^2\\
&=& (1+\tau^2)\langle\gamma'(u),J_1\gamma(u)\rangle^2=\frac{\tau^2}{1+\tau^2}\,,
\end{eqnarray*}
and we can choose $\vartheta$ such that $\tau^2/(1+\tau^2)=\sin^2\vartheta$.
Similarly,
\begin{eqnarray*}
g_{\tau}(E_1,F_u)&=&-\langle E_1,F_u\rangle+(1+\tau^2)\langle E_1,J_1F\rangle \langle F_u,J_1F\rangle\\
&=& \frac{\langle  F_u,J_1 F\rangle}{\tau}\left[1-(1+\tau^2)\right]\\
&=&(-\tau)\langle \gamma'(u),J_1\gamma(u)\rangle=\frac{\tau^2}{1+\tau^2}.
\end{eqnarray*}
\end{proof}

\begin{example}
If we take $\xi_2=\xi_3=\cst$, \eqref{eq-alpha123-b-0} becomes $\xi'(1-\sinh^2\xi_1)=0$. Thus, if also $\xi=\cst$, we
find, from \eqref{eq-descrizione-A}, a 1-parameter family $A(v)=A(\xi_1(v))$ of indefinite orthogonal matrices such that \eqref{eq-def-F-B-0}
defines a helix surface for any function $\xi_1$.
\end{example}

\section{The case $B>0$}
 Supposing $B>0$, integrating \eqref{eqquarta} we have the following
\begin{proposition}\label{prop-Fuv}
Let $M^2$ be a helix surface in $\S$ with constant angle $\vartheta$ so that $B>0$. Then, with respect to the local coordinates $(u,v)$ on $M$ defined in \eqref{eq:local-coordinates},  the position vector $F$ of $M^2$ in $\r^4_2$
is given by
$$
F(u,v)=\cos(\alpha_1\, u)\,g^1(v)+\sin(\alpha_1\, u)\,g^2(v)+\cos(\alpha_2\, u)\,g^3(v)+\sin(\alpha_2\, u)\,g^4(v),
$$
where
$$
\alpha_{1,2}=\frac{1}{\tau}(\tau \sqrt{B} \cos\vartheta\pm B)
$$
are positive real constants, while the $g^i(v)$, $i\in\{1,\dots,4\}$, are mutually orthogonal vector fields  in $\r^4_2$, depending only on $v$, such that
\begin{equation}\begin{aligned}\label{g1122}
g_{11}&=\langle g^1(v),g^1(v)\rangle=g_{22}=\langle g^2(v),g^2(v)\rangle=-\frac{\tau}{2B}\, \alpha_2\,,\\
g_{33}&=\langle g^3(v),g^3(v)\rangle=g_{44}=\langle g^4(v),g^4(v)\rangle=\frac{\tau}{2B}\, \alpha_1\,.
\end{aligned}\end{equation}
\end{proposition}
\begin{proof}
First, a direct integration of \eqref{eqquarta}, gives the solution
$$
F(u,v)=\cos(\alpha_1 u)\,g^1(v)+\sin(\alpha_1 u)\,g^2(v)+\cos(\alpha_2 u)\,g^3(v)+\sin(\alpha_2 u)\,g^4(v)\,,
$$
where
$$
\alpha_{1,2}=\sqrt{\frac{\tilde{b}^2-2\tilde{a}\pm\sqrt{\tilde{b}^4-4\tilde{a}\tilde{b}^2}}{2}}
$$
are two constants, while the $g^i(v)$, $i\in\{1,\dots,4\}$,  are vector fields in $\r^4_2$ which depend only on $v$. Now, taking into account the values of $\tilde{a}$ and $\tilde{b}$ given in \eqref{eq:value-a-b},  we get
\begin{equation}\label{alpha12}
\alpha_{1,2}=\frac{1}{\tau}(\tau \sqrt{B} \cos\vartheta\pm B)\,.
\end{equation}

Putting $g_{ij}(v)=\langle g^i(v),g^j(v)\rangle$, and evaluating  the relations \eqref{eq:Fprocuct}  in $(0,v)$, we obtain:
\begin{equation}\label{uno}
    g_{11}+g_{33}+2g_{13}=1,
\end{equation}
\begin{equation}\label{due}
    \alpha_1^2\,g_{22}+\alpha_2^2\,g_{44}+2\alpha_1\alpha_2\,g_{24}=\tilde{a},
\end{equation}
\begin{equation}\label{tre}
    \alpha_1\,g_{12}+\alpha_2\,g_{14}+\alpha_1\,g_{23}+\alpha_2\,g_{34}=0,
\end{equation}
\begin{equation}\label{quatro}
    \alpha_1^3\,g_{12}+\alpha_1\alpha_2^2\,g_{23}+\alpha_1^2\alpha_2\,g_{14}+\alpha_2^3g_{34}=0,
\end{equation}
\begin{equation}\label{cinque}
    \alpha_1^4\,g_{11}+\alpha_2^4\,g_{33}+2\alpha_1^2\alpha_2^2\,g_{13}=D,
\end{equation}
\begin{equation}\label{sei}
    \alpha_1^2\,g_{11}+\alpha_2^2\,g_{33}+(\alpha_1^2+\alpha_2^2)\,g_{13}=\tilde{a},
\end{equation}
\begin{equation}\label{sette}
    \alpha_1^4\,g_{22}+\alpha_1^3\alpha_2\,g_{24}+\alpha_1\alpha_2^3\,g_{24}+\alpha_2^4\,g_{44}=D,
\end{equation}
\begin{equation}\label{otto}
    \alpha_1^5\,g_{12}+\alpha_1^3\alpha_2^2\,g_{23}+\alpha_1^2\alpha_2^3\,g_{14}+\alpha_2^5\,g_{34}=0,
\end{equation}
\begin{equation}\label{nove}
    \alpha_1^3\,g_{12}+\alpha_1^3\,g_{23}+\alpha_2^3\,g_{14}+\alpha_2^3\,g_{34}=0,
\end{equation}
\begin{equation}\label{dieci}
    \alpha_1^6\,g_{22}+\alpha_2^6\,g_{44}+2\alpha_1^3\alpha_2^3\,g_{24}=E.
\end{equation}

From \eqref{tre}, \eqref{quatro}, \eqref{otto}, \eqref{nove}, it follows that $$g_{12}=g_{14}=g_{23}=g_{34}=0.$$
Also, from  \eqref{uno}, \eqref{cinque} and \eqref{sei}, we obtain
$$g_{11}=\frac{\tau^2\,(D+\alpha_2^4)+2B\sin^2\vartheta\,\alpha_2^2}{\tau^2(\alpha_1^2-\alpha_2^2)^2},\qquad g_{13}=0,\qquad
g_{33}=\frac{\tau^2\,(D+\alpha_1^4)+2B\sin^2\vartheta\,\alpha_1^2}{\tau^2(\alpha_1^2-\alpha_2^2)^2}.$$
Finally, using  \eqref{due}, \eqref{sette} and \eqref{dieci}, we obtain
$$g_{22}=\frac{\tau^2\,(E-2D\,\alpha_2^2)-B\sin^2\vartheta\,\alpha_2^4}{\tau^2\alpha_1^2\,(\alpha_1^2-\alpha_2^2)^2},\quad g_{24}=0,\qquad
g_{44}=\frac{\tau^2\,(E-2D\,\alpha_1^2)-B\sin^2\vartheta\,\alpha_1^4}{\tau^2\alpha_2^2\,(\alpha_1^2-\alpha_2^2)^2}.$$
We observe that $$g_{11}=g_{22}=\frac{\sqrt{B}-\tau\,\cos\vartheta}{2\sqrt{B}}<0,\qquad g_{33}=g_{44}=\frac{\sqrt{B}+\tau\,\cos\vartheta}{2\sqrt{B}}>0.$$
Therefore, taking into account \eqref{alpha12}, we obtain the expressions \eqref{g1122}.
\end{proof}

We are now in the right position to state the main result of this section.
\begin{theorem}\label{teo-principal1}
Let $M^2$ be a helix surface in the $\S\subset\r_2^4$ with constant angle $\vartheta\neq\pi/2$ so that $B>0$. Then, locally,
the position vector of $M^2$ in $\r^4_2$, with respect to the local coordinates $(u,v)$ on $M$ defined in \eqref{eq:local-coordinates}, is
\begin{equation}\label{eq-def-F-B-mag-0}
F(u,v)=A(v)\,\gamma(u)\,,
\end{equation}
where
\begin{equation}\label{eq-gamma-main-theorem}
\gamma(u)=(\sqrt{g_{33}}\,\cos (\alpha_2\, u),-\sqrt{g_{33}}\,\sin (\alpha_2\, u),\sqrt{-g_{11}}\,\cos (\alpha_1\, u),\sqrt{-g_{11}}\,\sin (\alpha_1\, u))
\end{equation}
is a curve in $\S$, $g_{11}$, $g_{33}$, $\alpha_1$, $\alpha_2$ are the four constants given in Proposition~\ref{prop-Fuv}, and $A(v)=A(\xi,\xi_1,\xi_2,\xi_3)(v)$ is a 1-parameter family of $4\times 4$ indefinite orthogonal matrices commuting with $J_1$, as described in \eqref{eq-descrizione-A}, with $\xi=\cst$ and
\begin{equation}\label{eq-alpha123}
  \cosh^2(\xi_1(v)) \,\xi_{2}'(v)+\sinh^2(\xi_1(v))\, \xi_{3}'(v)=0\,.
\end{equation}
Conversely, a parametrization $F(u,v)=A(v)\,\gamma(u)$, with $\gamma(u)$ and $A(v)$ as above, defines
 a constant angle surface in $\S$ with constant angle $\vartheta\neq\pi/2$.
\end{theorem}
\begin{proof}
With respect to the local coordinates $(u,v)$ on $M$ defined in \eqref{eq:local-coordinates}, Proposition~\ref{prop-Fuv} implies that the position vector of the  helix  surface in $\r^4_2$ is given by
$$
F(u,v)=\cos(\alpha_1 u)\,g^1(v)+\sin(\alpha_1 u)\,g^2(v)+\cos(\alpha_2 u)\,g^3(v)+\sin(\alpha_2 u)\,g^4(v)\,,
$$
where the vector fields  $\{g^i(v)\}_{i=1}^4$ are mutually orthogonal  and
$$
||g^1(v)||=||g^2(v)||=\sqrt{-g_{11}}=\text{constant}\,,
$$
$$
||g^3(v)||=||g^4(v)||=\sqrt{g_{33}}=\text{constant}\,.
$$
Thus, if we put $e_i(v)=g^i(v)/||g^i(v)||$, $i\in\{1,\dots,4\}$, we can write:
\begin{eqnarray}\label{eq:Fei}
F(u,v)=&\sqrt{-g_{11}}\,(\cos (\alpha_1\,u)\,e_1(v)+\sin(\alpha_1\,u)\,e_2(v))\nonumber\\
&+\sqrt{g_{33}}\,(\cos (\alpha_2\,u)\,e_3(v)+\sin(\alpha_2\,u)\,e_4(v))\,.
\end{eqnarray}

Now, the identities  \eqref{eq-fu-jf-main}, evaluated in  $(0,v)$,  become respectively:
\begin{equation}\label{eq1bis}\begin{aligned}
    &\alpha_2\,g_{33}\langle J_1e_3,e_4\rangle-\alpha_1 g_{11}\langle J_1e_1,e_2\rangle\\&+\sqrt{-g_{11}g_{33}}\,
    (\alpha_1\langle J_1e_3,e_2\rangle+\alpha_2\langle J_1e_1,e_4\rangle)
    =-\tau^{-1}\sin^2\vartheta,
    \end{aligned}
\end{equation}
\begin{equation}
     \langle J_1e_1,e_3\rangle=0\,,
\end{equation}
\begin{equation}\label{39}\begin{aligned}
&\alpha_2^3\,g_{33}\langle J_1e_3,e_4\rangle-\alpha_1^3\,g_{11}\langle J_1e_1,e_2\rangle\\&+
\sqrt{-g_{11}g_{33}}\,(\alpha_1\alpha_2^2\langle J_1e_3,e_2\rangle+\alpha_1^2\alpha_2\langle J_1e_1,e_4\rangle)=-I,
\end{aligned}
\end{equation}
\begin{equation}
    \langle J_1e_2,e_4\rangle=0\,,
\end{equation}
\begin{equation}\label{eq2bis}
    \alpha_1\langle J_1e_2,e_3\rangle+\alpha_2\langle J_1e_1,e_4\rangle=0\,,
\end{equation}
\begin{equation}\label{eq3bis}
    \alpha_2\langle J_1 e_2,e_3\rangle+\alpha_1\langle J_1e_1,e_4\rangle=0\,.
\end{equation}
We point out that  to obtain the previous identities we have divided by $\alpha_1^2-\alpha_2^2=4 \tau^{-1} \sqrt{B^3} \cos\vartheta$ which is, by the assumption on $\vartheta$, always different from zero.
From \eqref{eq2bis} and \eqref{eq3bis}, taking into account the $\alpha_1^2-\alpha_2^2\neq 0$, it results that
\begin{equation}\label{eq4bis}
     \langle J_1e_3,e_2\rangle=0\,,\qquad \langle J_1e_1,e_4\rangle=0\,.
\end{equation}
Therefore
$$
|\langle J_1e_1,e_2\rangle|=1=|\langle J_1e_3,e_4\rangle|.
$$
Substituting \eqref{eq4bis} in \eqref{eq1bis} and \eqref{39}, we obtain the system
\begin{equation}\label{j34}
\left\{\begin{aligned}\nonumber
&\alpha_1\, g_{11}\langle J_1e_1,e_2\rangle-\alpha_2\,g_{33}\langle J_1e_3,e_4\rangle=\tau^{-1}\sin^2\vartheta\\
 & \alpha_1^3\, g_{11}\langle J_1e_1,e_2\rangle-\alpha_2^3\,g_{33}\langle J_1e_3,e_4\rangle=I\,,
\end{aligned}
\right.
\end{equation}
a solution of which is
$$
\langle J_1e_1,e_2\rangle=\frac{\tau I-\alpha_2^2\sin^2\vartheta}{\tau g_{11}\,\alpha_1 (\alpha_1^2-\alpha_2^2)}\,,\qquad \langle J_1e_3,e_4\rangle=\frac{\tau I-\alpha_1^2\sin^2\vartheta}{\tau g_{33}\,\alpha_2(\alpha_1^2-\alpha_2^2)}\,.
$$
Now, as
$$
g_{11}\,g_{33}=-\frac{\sin^2\vartheta}{4B}\,,\qquad \alpha_1\,\alpha_2=\frac{B}{\tau^2}\sin^2\vartheta\,,\qquad \alpha_1^2-\alpha_2^2=\frac{4\sqrt{B^3}}{\tau}\cos\vartheta\,,
$$
it results that
$$
\langle J_1e_1,e_2\rangle\langle J_1e_3,e_4\rangle=1\,.
$$
Moreover, as $$\tau I-\alpha_2^2\sin^2\vartheta=2\tau^{-1}\,\sqrt{B^3}\cos\vartheta\sin^2\vartheta,$$  it results that $\langle J_1e_1,e_2\rangle<0$.
Consequently,
$\langle J_1e_1,e_2\rangle=\langle J_1e_3,e_4\rangle=-1$ and $J_1e_1=e_2$, $J_1e_3=-e_4$.

Then, if we fix the orthonormal basis  of $\r^4_2$ given by
$$
\tilde{E}_1=(0,0,1,0)\,,\quad \tilde{E}_2=(0,0,0,1)\,,\quad \tilde{E}_3=(1,0,0,0)\,,\quad \tilde{E}_4=(0,-1,0,0)\,,
$$
there must exists a 1-parameter family of $4\times 4$ indefinite orthogonal matrices $A(v)\in \mathrm{O}_2(4)$, with $J_1A(v)=A(v)J_1$,
such that $e_i(v)=A(v)\tilde{E}_i$.
Replacing $e_i(v)=A(v)\tilde{E}_i$ in \eqref{eq:Fei} we obtain
$$
F(u,v)=A(v)\gamma(u)\,,
$$
where
$$
\gamma(u)=(\sqrt{g_{33}}\,\cos (\alpha_2\, u),-\sqrt{g_{33}}\,\sin (\alpha_2\, u),\sqrt{-g_{11}}\,\cos (\alpha_1\, u),\sqrt{-g_{11}}\,\sin (\alpha_1\, u))\,
$$
is a curve in $\S$.

Let now examine the $1$-parameter family $A(v)$ that, according to \eqref{eq-descrizione-A},
depends on four functions $\xi_1(v),\xi_2(v),\xi_3(v)$ and $\xi(v)$. From \eqref{eq-definition-Fv}, it results that $\langle F_v, F_v\rangle=-\sin^2 \vartheta=\cst$. The latter implies that
\begin{equation}\label{eq-fv-fv-sin-theta-d-u}
\frac{\partial}{\partial u}\langle F_v, F_v\rangle_{| u=0}=0\,.
\end{equation}
Now, if we denote by ${\mathbf c_1},{\mathbf c_2},{\mathbf c_3},{\mathbf c_4}$ the four colons of $A(v)$, \eqref{eq-fv-fv-sin-theta-d-u} implies that
\begin{equation}\label{sistem-c23-c24}
\begin{cases}
\langle {\mathbf c_2}',{\mathbf c_3}'\rangle=0\\
\langle {\mathbf c_2}',{\mathbf c_4}'\rangle=0\,,
\end{cases}
\end{equation}
where with $'$ we means the derivative with respect to $v$.
Replacing in \eqref{sistem-c23-c24} the expressions of the ${\mathbf c_i}$'s as functions of $\xi_1(v),\xi_2(v),\xi_3(v)$ and $\xi(v)$, we obtain
\begin{equation}\label{sistemK-H}
\begin{cases}
\xi'\, h(v)=0\\
\xi'\, k(v)=0\,,
\end{cases}
\end{equation}
where $h(v)$ and $k(v)$ are two functions such that
$$
h^2+k^2=4 (\xi_1')^2+\sinh^2(2\xi_1)\, (-\xi'+\xi_2'+\xi_3')^2\,.
$$
From \eqref{sistemK-H} we have two possibilities:
\begin{itemize}
\item[(i)] $\xi=\cst$;
\item[] or
\item[(ii)] $4 (\xi_1')^2+\sinh^2(2\xi_1)\, (-\xi'+\xi_2'+\xi_3')^2=0$.
\end{itemize}
We will show that case (ii) cannot occur, more precisely we will show that if (ii) happens  then the parametrization $F(u,v)=A(v)\gamma(u)$ defines a Hopf tube, that is the Hopf vector field $E_1$ is tangent to the surface. To this end, we write the unit normal vector field $N$ as
$$
N=\frac{N_1 E_1+N_2E_2+N_3 E_3}{\sqrt{N_1^2+N_2^2+N_3^2}}\,.
$$
A long but straightforward computation (that can be also made using a software of symbolic computations) gives
$$\begin{aligned}
N_1&=1/2 (\alpha_1 + \alpha_2) \sqrt{-g_{11}} \sqrt{g_{33}}\,[2\xi_1'\,\cos(\alpha_1 u+\alpha_2 u-\xi_2+\xi_3)\\&+
  \sinh(2\xi_1)\sin(\alpha_1 u+\alpha_2 u-\xi_2+\xi_3)( -\xi'+\xi_2'+\xi_3')]\,.
  \end{aligned}
$$

Now case (ii) occurs if and only if $\xi_1=\cst=0$, or if $\xi_1=\cst\neq 0$ and $-\xi'+\xi_2'+\xi_3'=0$. In both cases $N_1=0$ and this implies that $g_{\tau}(N,J_1F)=-\tau g_{\tau}(N,E_1)=0$, i.e. the Hopf vector field is tangent to the surface. Thus we have proved that $\xi=\cst$. \\
Finally, in this case, \eqref{viceversa2} is equivalent to
$$
\tau \cos\vartheta\sqrt{B}[\cosh^2(\xi_1(v)) \,\xi_{2}'(v)+\sinh^2(\xi_1(v))\, \xi_{3}'(v)]=0\,.
$$
Since $\vartheta\neq\pi/2$ we conclude that condition \eqref{eq-alpha123} is satisfied.\\

The converse follows immediately from Proposition~\ref{pro-viceversa} since a direct
calculation shows that $g_{\tau}(F_u,F_u)=g_{\tau}(E_1,F_u)=\sin^2\vartheta$ which is \eqref{viceversa1},
while \eqref{eq-alpha123} is equivalent to \eqref{viceversa2}.
\end{proof}

\section{The case $B<0$}
In this section we study the case $B<0$. Integrating \eqref{eqquarta} we have the following result:
\begin{proposition}\label{prop-beta}
Let $M^2$ be a helix surface in $\S$ with constant angle $\vartheta$ and $B<0$. Then, with respect to the local coordinates $(u,v)$ on $M$ defined in \eqref{eq:local-coordinates},  the position vector $F$ of $M^2$ in $\r^4_2$
is given by
\begin{equation}\label{eqexpres}\begin{aligned}
F(u,v)&=\cos\big(\frac{\tilde{b}}{2}\, u\big)\,[\cosh (\beta\, u)\,w^1(v)+\sinh (\beta\, u)\,w^3(v)]\\&+
\sin\big(\frac{\tilde{b}}{2}\, u\big)\,[\cosh (\beta\, u)\,w^2(v)+\sinh (\beta\, u)\,w^4(v)],
\end{aligned}
\end{equation}
where
$$
\beta=\sqrt{-B}\,\cos\vartheta
$$
is a real constant, $\tilde{b}=-2 \tau^{-1}\, B$, while the $w^i(v)$, $i\in\{1,\dots,4\}$, are vector fields  in $\r^4_2$, depending only on $v$, such that
\begin{equation}\begin{aligned}\label{eq47}
\langle w^1(v),w^1(v)\rangle&=\langle w^2(v),w^2(v)\rangle=-\langle w^3(v),w^3(v)\rangle=-\langle w^4(v),w^4(v)\rangle=1,\\
\langle w^1(v),w^2(v)\rangle&=\langle w^1(v),w^3(v)\rangle=\langle w^2(v),w^4(v)\rangle=\langle w^3(v),w^4(v)\rangle=0,\\
\langle w^1(v),w^4(v)\rangle&=-\langle w^2(v),w^3(v)\rangle=-\frac{2\beta}{\tilde{b}}.
\end{aligned}
\end{equation}
\end{proposition}
\begin{proof}
A direct integration of \eqref{eqquarta}, gives the solution
$$\begin{aligned}
F(u,v)&=\cos\big(\frac{\tilde{b}}{2}\, u\big)\,[\cosh (\beta\, u)\,w^1(v)+\sinh (\beta\, u)\,w^3(v)]\\&+
\sin\big(\frac{\tilde{b}}{2}\, u\big)\,[\cosh (\beta\, u)\,w^2(v)+\sinh (\beta\, u)\,w^4(v)],
\end{aligned}
$$
where
$$
\beta=\frac{\sqrt{4\tilde{a}-\tilde{b}^2}}{2}=\sqrt{-B}\,\cos\vartheta
$$
is a constant, while the $w^i(v)$, $i\in\{1,\dots,4\}$,  are vector fields in $\r^4$ which depend only on $v$.
If $w_{ij}(v):=\langle w^i(v),w^j(v)\rangle$, evaluating the relations \eqref{eq:Fprocuct} in $(0,v)$, we obtain
\begin{equation}\label{uno2}
    w_{11}=1,
\end{equation}
\begin{equation}\label{due2}
    \frac{\tilde{b}^2}{4}\,w_{22}+\beta^2\,w_{33}+\beta\,\tilde{b}\,w_{23}=\tilde{a},
\end{equation}
\begin{equation}\label{tre2}
    \frac{\tilde{b}}{2}\,w_{12}+\beta \,w_{13}=0,
\end{equation}
\begin{equation}\label{quatro2}
     \frac{\tilde{b}}{2}\,\Big(\beta^2-\frac{\tilde{b}^2}{4}\Big)\,w_{12}+\beta^2 \,\tilde{b}\,w_{34}+\beta\,\frac{\tilde{b}^2}{2}\,w_{24}
     +\beta\,\Big(\beta^2-\frac{\tilde{b}^2}{4}\Big)\,w_{13}=0,
\end{equation}
\begin{equation}\label{cinque2}
   \Big(\beta^2-\frac{\tilde{b}^2}{4}\Big)^2\,w_{11}+\beta^2 \,\tilde{b}^2\,w_{44}
     +2\beta\,\tilde{b}\,\Big(\beta^2-\frac{\tilde{b}^2}{4}\Big)\,w_{14}=D,
\end{equation}
\begin{equation}\label{sei2}
   \Big(\beta^2-\frac{\tilde{b}^2}{4}\Big)\,w_{11}+\beta \,\tilde{b}\,w_{14}=-\tilde{a},
\end{equation}
\begin{equation}\label{sette2}
    \frac{\tilde{b}^2}{4}\, \Big(3\beta^2-\frac{\tilde{b}^2}{4}\Big)\,w_{22}+\beta^2 \,\Big(\beta^2-3\frac{\tilde{b}^2}{4}\Big)\,w_{33}
     +\beta\,\frac{\tilde{b}}{2}\,(4\beta^2-\tilde{b}^2)\,w_{23}=-D,
\end{equation}
\begin{equation}\label{otto2}\begin{aligned}
   &\frac{\tilde{b}}{2}\, \Big(3\beta^2-\frac{\tilde{b}^2}{4}\Big)\, \Big(\beta^2-\frac{\tilde{b}^2}{4}\Big)\,w_{12}+\tilde{b}\,\beta^2 \,\Big(\beta^2-3\frac{\tilde{b}^2}{4}\Big)\,w_{34}\\&
     +\beta\,\Big(\beta^2-3\frac{\tilde{b}^2}{4}\Big)\,\Big(\beta^2-\frac{\tilde{b}^2}{4}\Big)w_{13}+\beta\, \frac{\tilde{b}^2}{2}\, \Big(3\beta^2-\frac{\tilde{b}^2}{4}\Big)\,w_{24}=0,
     \end{aligned}
\end{equation}
\begin{equation}\label{nove2}
   \frac{\tilde{b}}{2}\,\Big(3\beta^2-\frac{\tilde{b}^2}{4}\Big)\,w_{12}+\beta\,\Big(\beta^2-3\frac{\tilde{b}^2}{4}\Big)\,w_{13}
  =0,
\end{equation}
\begin{equation}\label{dieci2}\begin{aligned}
   &\frac{\tilde{b}^2}{4}\, \Big(3\beta^2-\frac{\tilde{b}^2}{4}\Big)^2\,w_{22}+\beta^2 \,\Big(\beta^2-3\frac{\tilde{b}^2}{4}\Big)^2\,w_{33}\\&
     +\beta\,\tilde{b}\,\Big(3\beta^2-\frac{\tilde{b}^2}{4}\Big)\, \Big(\beta^2-3\frac{\tilde{b}^2}{4}\Big)\,w_{23}=E.
     \end{aligned}
\end{equation}
From  \eqref{uno2}, \eqref{cinque2} and \eqref{sei2}, it follows that
$$w_{11}=-w_{44}=1, \qquad w_{14}=-\frac{2\,\beta}{\tilde{b}}.$$
Also, from \eqref{tre2} and \eqref{nove2}, we obtain
$$w_{12}=w_{13}=0$$ and, therefore, from \eqref{quatro2} and \eqref{otto2},
$$w_{24}=w_{34}=0.$$
Moreover, using \eqref{due2}, \eqref{sette2} and \eqref{dieci2}, we get
$$w_{22}=-w_{33}=1, \qquad w_{23}=\frac{2\,\beta}{\tilde{b}}.$$
\end{proof}

\begin{theorem}\label{teo-principal2}
Let $M^2$ be a helix surface in  $\S$ with constant angle $\vartheta\neq\pi/2$ so that $B<0$. Then, locally,
the position vector of $M^2$ in $\r^4_2$, with respect to the local coordinates $(u,v)$ on $M$ defined in \eqref{eq:local-coordinates}, is given by
$$
F(u,v)=A(v)\,\gamma(u)\,,
$$
where the curve
$
\gamma(u)=(\gamma_1(u),\gamma_2(u),\gamma_3(u),\gamma_4(u)))
$
is given by
\begin{equation}\label{gamma}\left\{\begin{aligned}
\gamma_1(u)&=\cos\Big(\frac{\tilde{b}}{2}u\Big)\,\cosh (\beta\,u)-\frac{2\beta}{\tilde{b}}\,\sin\Big(\frac{\tilde{b}}{2}u\Big)\,\sinh (\beta\,u),\\
\gamma_2(u)&=\sin\Big(\frac{\tilde{b}}{2}u\Big)\,\cosh (\beta\,u)+\frac{2\beta}{\tilde{b}}\,\cos\Big(\frac{\tilde{b}}{2}u\Big)\,\sinh (\beta\,u),\\
\gamma_3(u)&=\frac{\sin\vartheta}{\sqrt{-B}}\,\cos\Big(\frac{\tilde{b}}{2}u\Big)\,\sinh (\beta\,u),\\
\gamma_4(u)&=\frac{\sin\vartheta}{\sqrt{-B}}\,\sin\Big(\frac{\tilde{b}}{2}u\Big)\,\sinh (\beta\,u)\,,\\
\end{aligned}
\right.
\end{equation}
$\beta=\sqrt{-B}\,\cos\vartheta$, $\tilde{b}=-2 \tau^{-1}\, B$ and $A(v)=A(\xi,\xi_1,\xi_2,\xi_3)(v)$ is a 1-~parameter family of $4\times 4$ indefinite orthogonal matrices anticommuting with $J_1$, as described in \eqref{eq-descrizione-A}, with
$\xi=\cst$ and
\begin{equation}\label{eq-alpha123second}
\begin{aligned}
&\sin\vartheta\,[2\cos(\xi_2(v)-\xi_3(v))\,\xi_1'(v)+(\xi_2'(v)+\xi_3'(v))\sin(\xi_2(v)-\xi_3(v))\,\sinh(2\xi_1(v))]\\
&-2\tau\cos\vartheta\,[\cosh^2(\xi_1(v)) \,\xi_{2}'(v)+\sinh^2(\xi_1(v))\, \xi_{3}'(v)] =0\,.
 \end{aligned}
\end{equation}
Conversely, a parametrization $F(u,v)=A(v)\,\gamma(u)$, with $\gamma(u)$ and $A(v)$ as above, defines
 a helix surface in $\S$ with constant angle $\vartheta\neq\pi/2$.
\end{theorem}
\begin{proof}
From \eqref{eq47}, we can define the following orthonormal basis in $\r^4_2$:
\begin{equation}\left\{\begin{aligned}
e_1(v)&=w^1(v),\\
e_2(v)&=w^2(v),\\
e_3(v)&=\frac{1}{\sin\vartheta}[\sqrt{-B}\,w^3(v)-\tau\cos\vartheta\,w^2(v)],\\
e_4(v)&=\frac{1}{\sin\vartheta}[\sqrt{-B}\,w^4(v)+\tau\cos\vartheta\,w^1(v)],
\end{aligned}
\right.
\end{equation}
with $\langle e_1,e_1\rangle=1=\langle e_2,e_2\rangle$ and $\langle e_3,e_3\rangle=-1=\langle e_4,e_4\rangle$.

Evaluating the identities \eqref{eq-fu-jf-main}
in  $(0,v)$, and taking into account that:
\begin{equation}
\begin{aligned}\nonumber
&F(0,v)=w^1(v)\,,\\
&F_u(0,v)=\frac{\tilde{b}}{2}\,w^2(v)+\beta\,w^3(v)\,,\\
&F_{uu}(0,v)=\Big(\beta^2-\frac{\tilde{b}^2}{4}\Big)\,w^1(v)+\beta\, \tilde{b}\,w^4(v)\,,\\
&F_{uuu}(0,v)=\frac{\tilde{b}}{2}\,\Big(3\beta^2 -\frac{\tilde{b}^2}{4})\,w^2(v)+\beta\,\Big(\beta^2-\frac{3}{4}\tilde{b}^2\Big)\,w^3(v)\,,\\
&F_{uuuu}(0,v)=\Big(\beta^4-\frac{3}{2}\beta^2\,\tilde{b}^2+\frac{\tilde{b}^4}{16}\Big)\,w^1(v)+
2\beta\,\tilde{b}\,\Big(\beta^2-\frac{\tilde{b}^2}{4}\Big)\,w^4(v)\,,
\end{aligned}
\end{equation}
we conclude that
\begin{equation}
\begin{aligned}\nonumber
\langle J_1w^3,w^4\rangle&=-\langle J_1w^1,w^2\rangle=1,\\
\langle J_1w^3,w^2\rangle&=\langle J_1w^1,w^4\rangle=0,\\
\langle J_1w^2,w^4\rangle&=\langle J_1w^1,w^3\rangle=-\frac{\tau\,\cos\vartheta}{\sqrt{-B}}.
\end{aligned}
\end{equation}
Then,
$$-\langle J_1e_1,e_2\rangle=\langle J_1e_3,e_4\rangle=1,$$
$$\langle J_1e_1,e_4\rangle=\langle J_1e_1,e_3\rangle=\langle J_1e_2,e_3\rangle=\langle J_1e_2,e_4\rangle=0.$$ Therefore, we obtain that $$J_1e_1=-e_2,\qquad J_1e_3=-e_4.$$
Consequently, if we consider the orthonormal basis $\{\hat{E}_i\}_{i=1}^4$ of $\r^4_2$ given by
$$
\hat{E}_1=(1,0,0,0)\,,\quad \hat{E}_2=(0,1,0,0)\,,\quad \hat{E}_3=(0,0,1,0)\,,\quad \hat{E}_4=(0,0,0,1)\,,
$$
there must exists a 1-parameter family of matrices $A(v)\in \mathrm{O}_2(4)$, with $J_1A(v)=-A(v)J_1$, such that $e_i(v)=A(v)\hat{E}_i$, $i\in\{1,\dots,4\}$.
As $$F=\langle F,e_1\rangle\,e_1+\langle F,e_2\rangle\,e_2-\langle F,e_3\rangle\,e_3-\langle F,e_4\rangle\,e_4,$$ computing $\langle F,e_i\rangle $ and substituting  $e_i(v)=A(v)\hat{E}_i$, we obtain that $
F(u,v)=A(v)\,\gamma(u),
$ where the curve $\gamma(u)$ of $\S$ is given in \eqref{gamma}.

Let now examine the $1$-parameter family $A(v)$ that, according to \eqref{eq-descrizione-A},
depends on four functions $\xi_1(v),\xi_2(v),\xi_3(v)$ and $\xi(v)$. Similarly to what we have done in the proof of Theorem~\ref{teo-principal1} we have that the condition
$$
\frac{\partial}{\partial u}\langle F_v, F_v\rangle_{| u=0}=0\,
$$
implies that
the functions  $\xi_1(v),\xi_2(v),\xi_3(v)$ and $\xi(v)$ satisfy the equation
$$\xi'\,[2\sin(\xi_2-\xi_3)\,\xi_1'-(\xi_2'+\xi_3'-\xi')\cos(\xi_2-\xi_3)\,\sinh(2\,\xi_1)]=0.$$
Then we have two possibilities:
\begin{itemize}
\item[(i)] $\xi=\cst$;
\item[] or
\item[(ii)] $2\sin(\xi_2-\xi_3)\,\xi_1'-(\xi_2'+\xi_3'-\xi')\cos(\xi_2-\xi_3)\,\sinh(2\,\xi_1)=0$.
\end{itemize}
Also in this case, using the same argument as in Theorem~\ref{teo-principal1},  condition (ii) would implies that the surface is a Hopf tube, thus we can assume
that $\xi=\cst$.

Finally, a long but straightforward computation shows that, in the case $\xi=\cst$, \eqref{viceversa2} is equivalent to \eqref{eq-alpha123second}.\\

The converse of the theorem follows immediately from Proposition~\ref{pro-viceversa} since a direct
calculation shows that $g_{\tau}(F_u,F_u)=g_{\tau}(E_1,F_u)=\sin^2\vartheta$ which is \eqref{viceversa1}
while \eqref{eq-alpha123second} is equivalent to \eqref{viceversa2}.
\end{proof}


\begin{thebibliography}{12}
\bibitem{B} B.~Daniel. Isometric immersions into $3$-dimensional homogeneous manifolds. {\em Comment. Math. Helv.} 82, (2007), 87--131.

\bibitem{CDS07} P.~Cermelli, A. J.~Di Scala. Constant-angle surfaces in liquid crystals. {\em Phil. Mag.} 87 (2007), 1871--1888.

\bibitem{DM09} F.~Dillen, M.I.~Munteanu. Constant angle surfaces in $\h^2\times\r$. {\em Bull. Braz. Math. Soc. (N.S.)} 40 (2009), 85--97.

\bibitem{DFVV07} F.~Dillen, J.~Fastenakels, J.~Van der Veken, L.~Vrancken. Constant angle surfaces in $\s^2\times\r$. {\em Monatsh. Math.} 152 (2007), 89--96.

\bibitem{DSRH10} A.~Di Scala, G.~Ruiz-Hern\'andez. Higher codimensional Euclidean helix submanifolds. {\em Kodai Math. J.} 33 (2010), 192--210.

\bibitem{DSRH09} A.~Di Scala, G.~Ruiz-Hern\'andez. Helix submanifolds of Euclidean spaces. {\em Monatsh. Math.} 157 (2009), 205--215.

\bibitem{FMV11} J.~ Fastenakels, M.I.~ Munteanu, J.~Van Der Veken. Constant angle surfaces in the Heisenberg group. {\em Acta Math. Sin. (Engl. Ser.)} 27 (2011), 747--756.

\bibitem{LM11} R.~L\'opez, M.I.~Munteanu. On the geometry of constant angle surfaces in $Sol_3$. {\em Kyushu J. Math.} 65 (2011), 237--249.

\bibitem{MO} S. Montaldo, I.I. Onnis, Helix surfaces in the Berger sphere. {\em Israel J. of Math.}, DOI:~10.1007/s11856-014-1055-6.

\bibitem{RH11} G.~ Ruiz-Hern\'andez. Minimal helix surfaces in $N^n\times\r$. {\em Abh. Math. Semin. Univ. Hambg.} 81 (2011), 55--67.

\bibitem{t} F.~Torralbo. Rotationally invariant constant mean curvature surfaces in homogeneous 3-manifolds. {\em Differential Geom. Appl.} 28 (2010), 593--607.

\end{thebibliography}
\end{document}